\documentclass{amsart}
\usepackage[totalwidth=480pt, totalheight=680pt]{geometry}
\usepackage{amsmath, amssymb, amsthm, amsfonts, mathrsfs, eucal, enumerate, rotating, lscape,layout}
\usepackage[numbers]{natbib}
\usepackage[all,tips,2cell]{xy}
\usepackage[usenames]{color}
\usepackage{xspace}
\usepackage{ifpdf}
\usepackage{ifthen}
\usepackage{hyperref}
\usepackage{paralist}
\newtheorem{theorem}{Theorem}[section]
\newtheorem*{theorem*}{Theorem}

\newtheorem{lemma}[theorem]{Lemma}
\newtheorem{proposition}[theorem]{Proposition}
\newtheorem{def-prop}[theorem]{Definition-Proposition}
\theoremstyle{remark}
\newtheorem{remark}[theorem]{Remark}
\theoremstyle{definition}

\theoremstyle{notation}

\theoremstyle{example}
\newtheorem{example}[theorem]{Example}
\numberwithin{equation}{section}
\newcommand{\Hom}{\mathrm{Hom}}
\newcommand{\Endo}{\mathrm{End}}

\newcommand{\id}{\mathrm{id}}
\newcommand{\boyut}[1]{\mathcal{#1}}

\newcommand{\oneC}{\boyut C}
\newcommand{\oneD}{\boyut D}

\newcommand{\oneH}{\boyut H}

\newcommand{\ra}{{\rightarrow}}

\newcommand{\la}{{\leftarrow}}
\newcommand{\Ra}{{\Rightarrow}}

\newcommand{\gact}{\mathcal{A}ct}

\newcommand{\grp}{\mathcal{G}rp}

\newcommand{\xmod}{\mathcal{X}\mathrm{Mod}}
\newcommand{\Xmod}{\mathbb{X}\mathrm{Mod}}

\newcommand{\AXMod}{\mathcal{X}\mathrm{Act}}

\newcommand{\Aut}{\mathrm{Aut}}
\newcommand{\AAut}{\mathcal{A}}

\newcommand{\Con}{\mathsf{C}}
\newcommand{\Wseq}{\mathcal{W}\mathrm{Seq}}
\newcommand{\Xsq}{\mathcal{X}\mathrm{Sq}}

\title{The category of crossed modules of crossed modules and its associated double groupoids}
\author{Nelson Martins-Ferreira}
\address{School of Technology and Management - ESTG\\Centre for rapid and Sustainable Product Developement-CDRSP, Polytechnic Institute of Leiria, P-2411-901, Leiria, Portugal}
\email{martins.ferreira@ipleiria.pt}
\author{Ahmet Emin Tatar}
\address{Department of Mathematics and Statistics, KFUPM, Dhahran, KSA}
\email{atatar@kfupm.edu.sa}
\date{}
\xyoption{2cell}
\begin{document}
\begin{abstract}
In this work we study the notion of Whitehead sequence in the category of crossed modules and actions of crossed modules. As expected, Whitehead sequences in that context are the same as crossed squares. We investigate under which conditions a Whitehead sequence of crossed modules gives rise to an internal groupoid in the category of crossed modules. In other words, we explicitly investigate the so called ``Smith is Huq" condition in the category of crossed modules.
\end{abstract}
\maketitle \UseAllTwocells
\section*{Introduction}
The notion of Whitehead sequence was introduced in \cite{MR3579003} to define internal groupoids, internal crossed modules, and internal actions in a fairly general context. The definition is designed in such a way that it is always possible to define a Whitehead sequence in a category $\oneC$ with respect to an action system of $\oneC$ over $\oneD$. Under that general assumptions, the following natural question arises: when does a Whitehead sequence have an associated groupoid structure? The answer to the question is provided by the following result: 

\begin{theorem}\label{thm_intro}(\cite{MR3579003}[Theorem 6.1])
Let $(I,R,J)$ be an action system of $\oneC$ over $\oneD$. If the pair of functors $(I,J)$ is jointly conservative and $\oneC$ is equivalent to the category of points in $\oneD$ in a way compatible with the system $(I,R,J)$, then the category of Whitehead sequences in $\oneC$ over $\oneD$ is equivalent to the category of internal groupoids in $\oneD$.
\end{theorem}

To motivate this result, let us analyze an example of a well known case.  In \cite{MR3579003}, it is shown that the category of Whitehead sequences in the category of group actions $\gact$ is equivalent to the category of crossed modules in the category of groups $\grp$. In the same paper, it is explained how to recover this equivalence from Theorem \ref{thm_intro}. Indeed, the functors
\begin{align*}
I: \xymatrix@1{\gact \ar[r] & \grp},& ~ (X,B) \mapsto B;\\
J: \xymatrix@1{\gact \ar[r]& \grp}, &~ (X,B) \mapsto X_{\circ}:=H;\\
R: \xymatrix@1{\grp \ar[r] & \gact}, &~ B \mapsto (\overline{B},B);
\end{align*}
equip $\gact$ with an action system over $\grp$ which satisfies the conditions of Theorem \ref{thm_intro}. Hence, the category of Whitehead sequences in $\gact$ is equivalent to the category of internal groupoids in the category of groups where the latter is know to be equivalent to the category of crossed modules by  Brown-Spencer Theorem \cite{BROWN1976296}.

A more general result is obtained by replacing $\grp$ with any semi-abelian category and by considering the category of its internal actions \cite{Borceux2005}. If the semi-abelian category \cite{MR1887164} satisfies the so called ``Smith is Huq"  condition \cite{MR2899723} (or equivalently, that every star multiplicative graph is multiplicative \cite{MR1990690}) then there is an equivalence between Whitehead sequences and internal crossed modules in the sense of \cite{MR1990690}.

Our purpose is to better understand what happens in the next dimension, that is, when the natural transformations between Whitehead sequences are taken into account. For the moment, to get some intuition in how to proceed, we study the particular case of crossed squares and associated double groupoids in the category of groups. We first define the category of actions of crossed modules $\AXMod$. An action of a crossed module over a crossed module is defined in \cite{MR1087375} and in \cite{Casas2010} in a larger framework. Here, we give a categorical definition: a collection $(X,(M,P,\mu))$ consisting of a crossed module $(M,P,\mu)$ and a 2-functor $X$ from $\underline{\underline{(M,P,\mu)}}$, the 2-category associated to $(M,P,\mu)$, to the 2-category of crossed modules $\Xmod$ describes an action of $(M,P,\mu)$ over the crossed module $X(\circ):=(T,G,\partial)$ where $\circ$ is the unique object of $\underline{\underline{(M,P,\mu)}}$. We define the category of Whitehead sequences, denoted by $\Wseq$, in $\AXMod$ over $\xmod$. We show that $\Wseq$ is equivalent to the category of crossed squares (Theorem \ref{thm:whitehead_crossed_square}). 

Later, in the paper we discuss whether we can obtain the latter equivalence from a bigger framework, that is as a consequence of Theorem \ref{thm_intro}. We show that we can equip $\AXMod$ with an action system over $\xmod$ given by the triple functors   
\begin{align}
 I: \xymatrix@1{\AXMod \ar[r] & \xmod},& ~ (X,(M,P,\mu)) \mapsto (M,P,\mu);\nonumber \\
\label{act_sys_intro}J: \xymatrix@1{\AXMod \ar[r]& \xmod}, &~ (X,(M,P,\mu)) \mapsto X(\circ):=(T,G,\partial);\\
 R: \xymatrix@1{\xmod \ar[r] & \AXMod}, &~ (M,P,\mu) \mapsto (\Con,(M,P,\mu)).\nonumber
\end{align}
where $(\Con,(M,P,\mu))$ is the conjugation action of $(M,P,\mu)$ on itself. We verify that the action system (\ref{act_sys_intro}) satisfies the conditions of Theorem \ref{thm_intro}. Hence, we deduce that $\Wseq$ is equivalent to the category of internal categories in the category of crossed modules, that is the category of 2-cat-groups which is known to be isomorphic to the category of crossed squares by \cite[Proposition 5.2]{LODAY1982179}. 

In this paper, we ignored the 2-cell structure of the crossed modules and of the actions of crossed modules. In a subsequent work, we will include these structures in to the discussion which will shed light on higher Whitehead sequences and eventually lead to the generalization of Theorem \ref{thm_intro} to internal bicategories as defined in \cite{Martins-Ferreira2017}.

This paper is organized as follows: In the first section, we recall the notions of crossed modules and the crossed squares in the category of groups. In the second section, we define the category of actions of crossed modules $\AXMod$. In the third section, we define the category of Whitehead sequences $\Wseq$ in $\AXMod$ and show that $\Wseq$ is equivalent to the category of crossed squares. In the fourth section, we define an action system of $\AXMod$ over $\xmod$. We show that this action system satisfies the conditions of Theorem \ref{thm_intro}. Therefore, Theorem \ref{thm:whitehead_crossed_square} is a consequence of Theorem \ref{thm_intro}.  


\section{Internal Groupoids in the Category of Crossed Modules}
\subsection{Crossed Modules}
We quickly recall crossed modules in groups and their 1- and 2-morphisms. A \emph{crossed module} is a group homomorphism $\partial: T \ra G$ where $G$ acts on $T$ and the action satisfies the conditions $\partial({}^{x}t)=x\partial(t)x^{-1}$ for all $x \in G$ and $t \in T$ and ${}^{\partial(t)}t'=tt't^{-1}$ for all $t,t' \in T$. A \emph{morphism of crossed modules} $(T_1,G_1,\partial_1) \ra (T_2,G_2,\partial_2)$ consists of a pair of group homomorphisms $f=(f_T:T_1\ra T_2,f_G:G_1 \ra G_2)$ so that $\partial_2 \circ f_T= f_G \circ \partial_1$ and $f_T({}^{x}t)={}^{f_G(x)}f_{T}(t)$ for all $x \in G_1$ and $t \in T_1$. We denote the category of crossed modules and their morphisms by $\xmod$.  A \emph{2-morphism of crossed modules} $\alpha:(f_T,f_G) \Ra (f'_T,f'_G)$ between two parallel morphisms of crossed modules from $(f_T,f_G),(f'_T,f'_G):(T_1,G_1,\partial_1) \ra (T_2,G_2,\partial_2)$ consists of a map $\alpha:G_1 \ra T_2$
\begin{equation*}\label{2_morphism_of_crossed_modules}
\begin{tabular}{c}
\xymatrix@!=1cm{T_1 \ar[d]_{\partial_1}\ar@<0.5ex>[r]^{f_T} \ar@<-0.5ex>[r]_{f'_T} & T_2 \ar[d]^{\partial_2}\\
G_1\ar@<0.5ex>[r]^{f_G} \ar@<-0.5ex>[r]_{f'_G} \ar@{-->}[ur]|{\alpha}& G_2}
\end{tabular}
\end{equation*}
satisfying the conditions
\begin{align}
\label{2_morphism_of_crossed_modules_1}f'_G(x)&=\partial_2(\alpha(x))f_G(x) \mathrm{~for~ all~} x \in G_1;\\
\label{2_morphism_of_crossed_modules_2} f_T'(t)&=\alpha(\partial_1(t))f_T(t), \mathrm{~for~ all~} t \in T_1;\\
\label{2_morphism_of_crossed_modules_3} \alpha(xx')&=\alpha(x)\!\ {}^{f_G(x)}\alpha(x')\mathrm{~for~ all~} x,x' \in G_1.
\end{align}

The 1- and 2-morphisms between any two crossed modules $(T_1,G_1,\partial_1)$ and $(T_2,G_2,\partial_2)$ form a category $\oneH((T_1,G_1,\partial_1),(T_2,G_2,\partial_2))$ under the \emph{vertical composition} depicted by the diagram
\begin{equation*}
\begin{tabular}{c}
\xymatrix@!=1.25cm{T_1 \ar[d]_{\partial_1}\ar@<0.75ex>[r]^{f_T} \ar[r]|{f'_T}\ar@<-0.75ex>[r]_{f''_T} & T_2 \ar[d]^{\partial_2} \ar@{}[dr]|{=}& T_1 \ar[d]_{\partial_1}\ar@<0.5ex>[r]^{f_T} \ar@<-0.5ex>[r]_{f''_T} & T_2 \ar[d]^{\partial_2}\\
G_1\ar@<0.75ex>[r]^{f_G} \ar@<-0.75ex>[r]_{f''_G} \ar[r]|{f'_G}\ar@{-->}@<-0.5ex>[ur]^{\alpha} \ar@{-->}@<0.5ex>[ur]_{\beta}& G_2 &G_1\ar@<0.5ex>[r]^{f_G} \ar@<-0.5ex>[r]_{f''_G} \ar@{-->}[ur]|{\beta \circ \alpha}& G_2 }
\end{tabular}
\end{equation*}
and defined by the formula $\beta \circ \alpha:=\beta\alpha$. 

For any three crossed modules $(T_1,G_1,\partial_1)$, $(T_2,G_2,\partial_2)$, and $(T_3,G_3,\partial_3)$, we can define a functor
\begin{equation*}
\xymatrix@1{\oneH((T_1,G_1,\partial_1),(T_2,G_2,\partial_2)) \times \oneH((T_2,G_2,\partial_2),(T_3,G_3,\partial_3)) \ar[r] & \oneH((T_1,G_1,\partial_1),(T_3,G_3,\partial_3))},
\end{equation*}
called the \emph{horizontal composition} depicted by the diagram 

\begin{equation}\label{horizontal_composition_crossed_modules}
\begin{tabular}{c}
\xymatrix@!=1cm{T_1 \ar[d]_{\partial_1}\ar@<0.5ex>[r]^{f_T} \ar@<-0.5ex>[r]_{f'_T} & T_2 \ar[d]_{\partial_2}\ar@<0.5ex>[r]^{g_T} \ar@<-0.5ex>[r]_{g'_T} & T_3 \ar[d]^{\partial_3} \ar@{}[dr]|{=}&T_1 \ar[d]_{\partial_1}\ar@<0.5ex>[r]^{g_T \circ f_T} \ar@<-0.5ex>[r]_{g'_T \circ f'_T} & T_3 \ar[d]^{\partial_3} \\
G_1\ar@<0.5ex>[r]^{f_G} \ar@<-0.5ex>[r]_{f'_G} \ar@{-->}[ur]_{\alpha}& G_2\ar@<0.5ex>[r]^{g_G} \ar@<-0.5ex>[r]_{g'_G} \ar@{-->}[ur]_{\beta}& G_3 & G_1\ar@<0.5ex>[r]^{g_G \circ f_G} \ar@<-0.5ex>[r]_{g'_G \circ f'_G} \ar@{-->}[ur]|{\beta * \alpha}& G_3}
\end{tabular}
\end{equation}
and defined by the formula 
\begin{equation}\label{horizontal_composition}
\beta*\alpha:=(\beta \circ f'_G)(g_T \circ \alpha).
\end{equation}
\begin{lemma}\label{2nd_def_of_horizontal_composition}
In the setting of the diagram (\ref{horizontal_composition_crossed_modules}), the horizontal composition can be also defined by the relation $\beta*\alpha:=(g'_T \circ \alpha)(\beta \circ f_G)$.
\end{lemma}
\begin{proof}
As 2-morphisms $\alpha: (f_T,f_G) \ra (f'_T,f'_G)$ and $\beta:(g_T,g_G) \ra (g'_T,g'_G)$ respectively satisfy the relations $\partial_2(\alpha(x))f_G(x)=f'_G(x)$ and $\beta(\partial_2(t))g_T(t)=g'_T(t)$. Then, for any $x \in G_1$
\begin{align*}
(\beta\circ f'_G(x))(g_T \circ \alpha(x))&=\beta\left(\partial_2(\alpha(x))f_G(x)\right)(g_T \circ \alpha(x))\\
&=\beta(\partial_2(\alpha(x))) \!\ {}^{g_G(\partial_2(\alpha(x))}(\beta \circ f_G(x))(g_T \circ \alpha(x))\\
&=\beta(\partial_2(\alpha(x))) \!\ {}^{\partial_3(g_T(\alpha(x)))}(\beta \circ f_G(x))(g_T \circ \alpha(x))\\
&=\beta(\partial_2(\alpha(x))) \!\ g_T(\alpha(x))(\beta \circ f_G(x))\\
&=(g'_T \circ \alpha(x))(\beta \circ f_G(x))
\end{align*}
\end{proof}

Crossed modules and their 1- and 2- morphisms with the vertical and horizontal compositions defined above form the \emph{2-category of crossed modules} denoted by $\Xmod$.

A crossed module $(T,G,\partial)$ can be associated a groupoid $\underline{(T,G,\partial)}$ whose objects are the elements of $G$, and whose morphisms are the elements of $T \times G$ where a pair $(t,x) \in T \times G$ represents a morphism from $x$ to $\partial(t)x$. The composition is given by the group operation in $T$, that is if $(t_2,x_2) \circ (t_1,x_1)=(t_2t_1,x_1)$ given that $x_2=\partial(t_1)x_1$. Moreover, there exists a bifunctor
\begin{equation}\label{group_like_structure_functor}
\otimes:\xymatrix@1{\underline{(T,G,\partial)} \times \underline{(T,G,\partial)} \ar[r] & \underline{(T,G,\partial)}},
\end{equation}
given by the group operation of $G$ on objects and by the semidirect product of the group $T \rtimes G$ on the morphisms. Verifying that this functor preserves composition is a simple exercise which requires using the axioms of the crossed module $(T,G,\partial)$. Hence, we can see $\underline{(T,G,\partial)}$ as a 2-category $\underline{\underline{(T,G,\partial)}}$ with one object whose hom-category is $\underline{(T,G,\partial)}$ and whose composition functor is (\ref{group_like_structure_functor}). From another perspective, we can consider the existence of (\ref{group_like_structure_functor}) as a strict group-like structure on $\underline{(T,G,\partial)}$ in the sense of Breen \cite{MR1191733}.

This construction is functorial. We can associate to a crossed module morphism $(f_T,f_G):(T_1,G_1,\partial_1) \ra (T_2,G_2,\partial_2)$ a functor $F:\underline{(T_1,G_1,\partial_1)} \ra \underline{(T_2,G_2,\partial_2)}$ given by $F(x):=f_G(x)$ for every $x \in G$ object in $\underline{(T_1,G_1,\partial_1)}$ and by $F(t,x):=(f_T(t),f_G(x))$ for every $(t,x) \in T \rtimes G$ morphism in $\underline{(T_1,G_1,\partial_1)}$. We remark that for every $x_1,x_2 \in G$ objects in $\underline{(T_1,G_1,\partial_1)}$, $F(x_2x_1)=F(x_2)F(x_1)$ and for every $(t_1,x_1),(t_2,x_2) \in T \rtimes G$ morphisms in $\underline{(T_1,G_1,\partial_1)}$, $F((t_2,x_2)(t_1,x_1))=F(t_2,x_2)F(t_1,x_1)$. That is, $F$ respects the group-like structures on the categories. We call such functors morphisms of group-like categories or simply additive functors. As $F$ preserves the group-like structure not just up to an isomorphism but on the nose, we call it strict additive functor. 

Reciprocally, starting with a strict additive functor $F:\underline{(T_1,G_1,\partial_1)} \ra \underline{(T_2,G_2,\partial_2)}$, we construct a morphism of crossed modules $(f_T,f_G):(T_1,G_1,\partial_1) \ra (T_2,G_2,\partial_2)$ as follows: We define $f_G:G_1 \ra G_2, ~f_G(x)=F(x)$. Before defining $f_T:T_1\ra T_2$, observe that as $F$ is strict it sends any morphism that emanates from 1 to such a morphism, that is, for any $t \in T_1$, $F(t,1)$ is of the form $(t',1)$ where $t' \in T_2$. Hence, we define $f_T(t)=t'$. Due to the strictness of $F$, $f_G$ and $f_T$ are group homomorphisms. The definitions of $f_G$ and $f_T$ also satisfy the condition $f_G \circ \partial_1=\partial_2 \circ f_T$. In fact, $F(t,1)$ which is by definition the morphism $1 \ra \partial_2(f_T(t))$ is also the morphism $1\ra F(\partial_1(t))$. Hence,  $f_G \circ \partial_1(t)=\partial_2 \circ f_T(t)$, for every $t \in T$. Finally, $f_G$ and $f_T$ respect the actions in the sense that for every $x \in G_1$ and $t \in T_1$, we have $f_T({}^{x}t)={}^{f_G(x)}f_T(t)$:
\begin{align*}
(f_T({}^{x}t),1)&=F({}^x{t},1)=F\left((1,x)(t,1)(1,x)^{-1}\right)=F(1,x)F(t,1)F(1,x)^{-1}\\
&=(1,f_G(x))(f_T(t),1)(1,f_G(x))^{-1}=({}^{f_G(x)}f_T(t),1).
\end{align*}

We shall note that the strictness of the additive functor $F:\underline{(T_1,G_1,\partial_1)} \ra \underline{(T_2,G_2,\partial_2)}$ is the key ingredient in the construction of the morphism of crossed squares $(f_T,f_G):(T_1,G_1,\partial_1) \ra (T_2,G_2,\partial_2)$. If $F$ is not strict, then the associated crossed module morphism $\underline{(T_1,G_1,\partial_1)} \ra \underline{(T_2,G_2,\partial_2)}$ is not strict either, that is it is not a commutative square but rather a butterfly \cite{noohi-2005}.

\subsection{Crossed Squares as Double Groupoids}
We remind from \cite{LODAY1982179} the definition of crossed squares and some facts about them. A \emph{crossed square} is a commutative diagram of groups and group homomorphisms
\begin{equation}\label{crossed_square_prep}
\begin{tabular}{c}
\xymatrix@!=1cm{T \ar[r]^{f} \ar[d]_{\partial} & M \ar[d]^{\mu}\\G \ar[r]_{g}&P}
\end{tabular}
\end{equation}
equipped with actions of $P$ on $T$, on $M$, and on $G$, a function $\phi:M \times G \ra T$ satisfying the axioms
\begin{enumerate}[(CS-1)]
\item $f$ and $\partial$ are $P$-equivariant;
\item $\mu$, $g$, and $\mu \circ f=g \circ \partial$ are crossed modules;
\item $f \circ \phi (m,x)=m\!\ {}^{g(x)}m^{-1}$ for every $m \in M$, $x \in G$;
\item $\partial \circ \phi(m,x)={}^{\mu(m)}x\!\ x^{-1}$ for every $m \in M$, $x \in G$;
\item $\phi(f(t),x)=t\!\ {}^{g(x)}t^{-1}$ for every $t \in T$, $x \in G$;
\item $\phi(m,\partial(t))={}^{\mu(m)}t\!\ t^{-1}$ for every $t \in T$, $m \in M$;
\item $\phi(m_0m_1,x)={}^{\mu(m_0)}\phi(m_1,x)\phi(m_0,x)$ for every $m_0,m_1 \in M$, $x \in G$;
\item $\phi(m,x_0x_1)=\phi(m,x_0)\!\ {}^{g(x_0)}\phi(m,x_1)$ for every $m \in M$, $x_0,x_1 \in G$;
\item ${}^{p}\phi(m,x)=\phi({}^pm,{}^px)$, for every $m \in M$, $x \in G$, and $p \in P$.
\end{enumerate}
It follows from the above axioms that the homomorphisms $\partial$ and $f$ are crossed modules, as well.  A \emph{morphism of crossed squares} is given by the horizontal morphisms in the commutative diagram
\begin{equation}\label{morphism_crossed_squares}
\begin{tabular}{c}
\xymatrix@!=0.5cm{&T_1\ar[rr]\ar[dr]^{f_1} \ar[dl]_{\partial_1} &&T_2\ar[dr]^{f_2}&&&&T_1\ar[rr] \ar[dl]_{\partial_1} &&T_2\ar[dr]^{f_2}\ar[dl]_{\partial_2}&\\
G_1 \ar[dr]_{g_1}&&M_1 \ar[rr] \ar[dl]^{\mu_1}&&M_2\ar[dl]^{\mu_2}&=&G_1 \ar[dr]_{g_1} \ar[rr]&& G_2 \ar[dr]_{g_2} &&M_2 \ar[dl]^{\mu_2}\\
&P_1 \ar[rr] && P_2 &&&& P_1 \ar[rr] &&P_2&}
\end{tabular}
\end{equation} 
that are compatible with the actions and the maps $\phi_1$ and $\phi_2$. $\Xsq$ denotes the category of crossed squares .

\section{Category of Actions of Crossed Modules}
In this section, we define the category of actions of crossed modules denoted by $\AXMod$. Automorphisms and actions of crossed modules are known since \cite{MR1087375}. Let us recollect some of these definitions and facts related to them. 

A \emph{derivation} is a map $\chi: G \ra T$ which satisfies $\chi(xy)=\chi(x)\!\ {}^{x}\chi(y)$ for all $x,y \in G$. Any derivation $\chi$ defines a pair of endomorphisms $\sigma:G \ra G, \,\ x \mapsto \partial(\chi(x))x$ and $\theta:T \ra T, \,\ t \mapsto \chi(\partial(t))t$. Two derivations $\chi_1$ and $\chi_2$ can be multiplied and the product is the derivation $\chi$ denoted by $\chi_1 * \chi_2$ and defined by the formula
\begin{equation}\label{der_multiplication}
\chi(x)=\chi_1(x)*\chi_2(x)=\chi_1(\sigma_2(x))\chi_2(x) ~ (=\theta_1(\chi_2(x))\chi_1(x)),
\end{equation}
where $(\theta_1,\sigma_1)$ and $(\theta_2,\sigma_2)$ are pairs of endomorphisms which correspond respectively to $\chi_1$ and $\chi_2$. Under this multiplication, derivations form a semigroup. For future use, we shall record the following relations that $\theta$, $\sigma$, and $\chi$ satisfy:
\begin{align}
\label{der_1} \sigma(\partial(t))&=\partial(\theta(t));\\
\label{der_2} \theta(\chi(x))&= \chi(\sigma(x)); \\
\label{der_3} \theta({}^xt)&={}^{\sigma(x)}\theta(t).
\end{align}
A \emph{regular derivation} is a derivation which is a unit in the semigroup of derivations. The group of regular derivations is denoted by $D(G,T)$. In case the derivation $\chi$ is regular, the endomorphisms $\theta$ and $\sigma$ as defined above become automorphisms (see \cite{MR535788}), in fact, an automorphism of the crossed module $(T,G,\partial)$ due to (\ref{der_3}). Hence, there exists a group homomorphism $\Delta: D(G,T) \ra \Aut(T,G,\partial)$ where $\Aut(T,G,\partial)$ is the group of automorphisms of $(T,G,\partial)$. $\Delta$ assigns to any regular derivation $\chi$, the pair $(\theta,\sigma)$. The \emph{actor crossed module} $\AAut(T,G,\partial)$ of $(T,G,\partial)$ is the group homomorphism $\Delta$ with the action of $\Aut(T,G,\partial)$ on $D(G,T)$ given by ${}^{(f_T,f_G)}\chi:=f_T \circ \chi \circ f_G^{-1}$. 

Let $(M,P,\mu)$ and $(T,G,\partial)$ be two crossed modules. In  \cite{MR1087375}, it is defined that $(M,P,\mu)$ acts on $(T,G,\partial)$ if there exists a morphism of crossed modules from $(M,P,\mu)$ to the actor crossed module $\AAut(T,G,\partial)$
\begin{equation}\label{action_of_crossed_modules}
\begin{tabular}{c}
\xymatrix@!=1cm{M \ar[d]_{\mu} \ar[r]^{\varepsilon} & **[r]D(G,T)\ar[d]^{\Delta}\\P \ar[r]_{\rho} & **[r]\Aut(T,G,\partial)}
\end{tabular}
\end{equation}

To be able to define the morphisms of actions of crossed modules, hence their category,  we shall study the above definition of action of crossed modules from a categorical point of view. First of all, we notice that a regular derivation $\chi$ is nothing but a 2-morphism of crossed modules from the identity morphism $(\id_{T},\id_{G}):(T,G,\partial) \ra (T,G,\partial)$ to the morphism $(\theta,\sigma):(T,G,\partial) \ra (T,G,\partial)$.
\begin{equation*}\label{diag:regular_derivation}
\begin{tabular}{c}
\xymatrix@!=1cm{T \ar[d]_{\partial}\ar@<0.5ex>[r]^{\id_T} \ar@<-0.5ex>[r]_{\theta} & T \ar[d]^{\partial}\\
G\ar@<0.5ex>[r]^{\id_G} \ar@<-0.5ex>[r]_{\sigma} \ar@{-->}[ur]_{\chi}& G}
\end{tabular}
\end{equation*}
and the group homomorphism $\Delta$ simply maps $\chi$ to its codomain. Moreover, the multiplication of regular derivations $\chi_1$ and $\chi_2$ given by the relation (\ref{der_multiplication}) correspond to their horizontal composition described by (\ref{horizontal_composition}).
\begin{equation*}\label{diag:horizontal_composition}
\begin{tabular}{c}
\xymatrix@!=1cm{T \ar[d]_{\partial}\ar@<0.5ex>[r]^{\id_T} \ar@<-0.5ex>[r]_{\theta_2} & T \ar[d]^{\partial} \ar@<0.5ex>[r]^{\id_T} \ar@<-0.5ex>[r]_{\theta_1} & T \ar[d]^{\partial} \ar@{}[dr]|{=} & T \ar[d]_{\partial}\ar@<0.5ex>[r]^{\id_T} \ar@<-0.5ex>[r]_{\theta_2\circ \theta_1} & T \ar[d]^{\partial} \\
G\ar@<0.5ex>[r]^{\id_G} \ar@<-0.5ex>[r]_{\sigma_2} \ar@{-->}[ur]_{\chi_2}& G\ar@<0.5ex>[r]^{\id_G} \ar@<-0.5ex>[r]_{\sigma_1} \ar@{-->}[ur]_{\chi_1} & G & G\ar@<0.5ex>[r]^{\id_G} \ar@<-0.5ex>[r]_{\sigma_1 \circ \sigma_2} \ar@{-->}[ur]|{\chi_1 * \chi_2}& G}
\end{tabular}
\end{equation*}
After these observations, we are ready to define actions of crossed modules in a more categorical way: 
\begin{proposition}\label{prop:action_as_functor}
If $(M,P,\mu)$ acts on $(T,G,\partial)$ then there exists a strict 2-functor 
\begin{equation}\label{action_as_2_functor}
X:\xymatrix@1{\underline{\underline{(M,P,\mu)}} \ar[r] & \Xmod,}
\end{equation}
so that $X$ assigns to the only object of $\underline{\underline{(M,P,\mu)}}$ the crossed module $(T,G,\partial)$. Moreover, if there exists a strict 2-functor (\ref{action_as_2_functor}), then $(M,P,\mu)$ acts on the crossed module $(T,G,\partial)$ assigned to the only object of $\underline{\underline{(M,P,\mu)}}$.
\end{proposition}
\begin{proof}
Assume the action of $(M,P,\mu)$ on $(T,G,\partial)$ is given by the diagram (\ref{action_of_crossed_modules}). We define (\ref{action_as_2_functor}) as follows:
\begin{itemize}
\item The only object of $\underline{\underline{(M,P,\mu)}}$ is assigned to $(T,G,\partial)$.
\item The functor between the hom-categories
\begin{equation}\label{hom_categories_of_actions}
X_{\circ}:\xymatrix@1{\underline{(M,P,\mu)} \ar[r] & \Endo(T,G,\partial)},
\end{equation}
assigns to every 
\begin{itemize}
\item object $p \in P$ in $\underline{(M,P,\mu)}$, the automorphism $X_{\circ}(p):=\rho(p)=(\theta_p,\sigma_p)$ of the crossed module $(T,G,\partial)$,
\item morphism $(m,p) \in M \rtimes P$ in $\underline{(M,P,\mu)}$, the morphism of automorphisms $X_{\circ}(m,p):(\theta_p,\sigma_p) \Ra (\theta_{\mu(m)p},\sigma_{\mu(m)p})$ given by the map $\varepsilon(m) \circ \sigma_p: G \ra T$.
\end{itemize}
$X_{\circ}$ sends the identity morphism $(1,p)$ to $\varepsilon(1) \circ \sigma_p$. As $\varepsilon$ is a group homomorphism, for any $x \in G$, $\varepsilon(1)(\sigma_p(x))=1$, that is $\varepsilon(1)\circ \sigma_p$ is the identity 2-morphism $(\theta_p,\sigma_p) \Ra (\theta_{p},\sigma_{p})$. Let $(m,p)$ and $(m',p')$ be two morphisms of $\underline{(M,P,\mu)}$. Their composition is given by the semidirect product $(m'\!\ {}^{p'}m,p'p)$. To prove that $X_{\circ}$ preserves the composition, we shall prove the relation $X_{\circ}(m'\!\ {}^{p'}m,p'p)=X_{\circ}(m',p')*X_{\circ}(m,p)$. In fact;
\begin{align*}
X_{\circ}(m'\!\ {}^{p'}m,p'p)&=\varepsilon(m'\!\ {}^{p'}m)\circ \sigma_{p'}\circ\sigma_{p}\\
&=\left(\varepsilon(m')*\varepsilon({}^{p'}m)\right)\circ \sigma_{p'} \circ \sigma_{p}\\ 
&=\left(\left(\varepsilon(m') \circ \sigma_{\mu({}^{p'}m)}\right)\varepsilon({}^{p'}m)\right)\circ \sigma_{p'} \circ \sigma_{p}\\ 
&=\left(\varepsilon(m') \circ \sigma_{\mu({}^{p'}m)}\circ \sigma_{p'} \circ \sigma_{p}\right)\left(\varepsilon({}^{p'}m)\circ \sigma_{p'} \circ \sigma_{p}\right)\\ 
&=\left(\varepsilon(m') \circ \sigma_{p'}\circ \sigma_{\mu(m)p}\right)\left({}^{(\theta_{p'},\sigma_{p'})}\varepsilon(m)\circ \sigma_{p'} \circ \sigma_{p}\right)\\ 
&=\left(\varepsilon(m') \circ \sigma_{p'}\circ \sigma_{\mu(m)p}\right)\left(\theta_{p'} \circ \varepsilon(m)\circ \sigma_{p}\right)\\
&=X_{\circ}(m',p') * X_{\circ}(m,p)
\end{align*}
\end{itemize}
We note that $X$ is strict since on objects $X_{\circ}$ is defined by the group homomorphism $\rho$.

Reciprocally, given a strict 2-functor (\ref{action_as_2_functor}) that assigns $(T,G,\partial)$ to the only object of $\underline{\underline{(M,P,\mu)}}$, we define $\varepsilon: M \ra D(G,T), ~\varepsilon(m):=X_{\circ}(m,1)$ and $\rho:P \ra \Aut(T,G,\partial), ~ \rho(p):=X_{\circ}(p)$. The fact that $\varepsilon$ and $\rho$ are group homomorphisms follows from $X_{\circ}$ being a strict functor. $(m,1)$ is the morphism $1 \ra \mu(m)$ in $\underline{(M,P,\mu)}$ mapped to $(\id_T,\id_G) \Ra X_{\circ}(\mu(m))$ by $X_{\circ}$. Then from definitions, $\Delta \circ \varepsilon(m)$ is the codomain of $(\id_T,\id_G) \ra X_{\circ}(\mu(m))$, that is $\Delta \circ \varepsilon(m)=X_{\circ}(\mu(m))$ which is simply $\rho(\mu(m))$ by its definition. We show that $\varepsilon({}^{p}m)={}^{\rho(p)}\varepsilon(m)$ for any $(m,p) \in M \rtimes P$.
\begin{align*}
\varepsilon({}^{p}m)&=X_{\circ}({}^pm,1)\\
&=X_{\circ}\left((1,p)(m,1)(1,p^{-1})\right)\\
&=X_{\circ}(1,p) * X_{\circ}(m,1) * X_{\circ}(1,p^{-1})\\
&=\theta_p \circ X_{\circ}(m,1) \circ \sigma_p^{-1}\\
&={}^{\rho(p)}\varepsilon(m)
\end{align*} 
where $X_{\circ}(1,p)$ is the identity morphism of the automorphism $(\theta_p,\sigma_p)$.
\end{proof}

In the rest of the paper, we work with the categorical definition of actions of crossed modules and we represent an action by $(X,(M,P,\mu))$. 

\begin{example}\label{ex:conjugation}
As an example and for the future use, we remind the conjugation action of the crossed module $(M,P,\mu)$ on itself. $(M,P,\mu)$ acts by conjugation on itself if there exists a crossed module homomorphism
\begin{equation*}\label{conjugation_action_of_crossed_modules}
\begin{tabular}{c}
\xymatrix@!=1cm{M \ar[d]_{\mu} \ar[r]^{\eta} & **[r]D(P,M)\ar[d]^{\Delta}\\P \ar[r]_{\zeta} & **[r]\Aut(M,P,\mu)}
\end{tabular}
\end{equation*}
where $\eta(m)(p)=m\!\ {}^pm^{-1}$ and $\zeta(p)=(\psi_p,\varphi_p)$ with $\psi_p(m)={}^pm$ and $\varphi_p(p')=pp'p^{-1}$. We represent this action by the collection $(\Con,(M,P,\mu))$ where the 2-functor $\Con$ sends the unique object $\circ$ of $\underline{\underline(M,P,\mu)}$ to $(M,P,\mu)$ and the hom-functor $\Con_{\circ}: \underline{(M,P\mu)}\ra \Endo(M,P,\mu)$ is given for any $p \in P$ by $\Con_{\circ}(p)=\zeta(p)$ and for any $(m,p) \in M \rtimes P$ by $\Con_{\circ}(m,p)=\eta(m)\circ \varphi_p$.
\end{example}

Let $(X_1,(M_1,P_1,\mu_1))$ and $(X_2,(M_2,P_2,\mu_2))$ be two actions of crossed modules. A \emph{morphism of actions} 
\begin{equation}\label{1-morphism_of_actions_1}
(N,(\alpha_M,\alpha_P)):\xymatrix@1{(X_1,(M_1,P_1,\mu_1)) \ar[r] &(X_2,(M_2,P_2,\mu_2))},
\end{equation}
is given by 
\begin{itemize}
\item a morphism of crossed modules $(\alpha_M,\alpha_P): (M_1,P_1,\mu_1) \ra (M_2,P_2,\mu_2)$
\item a strict transformation $N: X_1 \Ra X_2 \circ A$, where $A$ is the strict 2-functor associated to $(\alpha_M,\alpha_P)$.
\end{itemize}
If we unfold this definition, the morphism (\ref{1-morphism_of_actions_1}) is given by the two morphisms of crossed modules
\[(\alpha_M,\alpha_P): (M_1,P_1,\mu_1) \ra (M_2,P_2,\mu_2),\]
\[N_{\circ}:=(N_{T},N_{G}):(T_1, G_1,\partial_1) \ra (T_2,G_2,\partial_2),\] 
so that the diagram
\begin{equation}\label{1-morphism_of_actions_2}
\begin{tabular}{c}
\xymatrix{\underline{(M_1,P_1,\mu_1)} \ar[r]^{X_{1\circ}} \ar[d]_{X_{2\circ} \circ A} & \AAut(T_1,G_1,\partial_1) \ar[d]^{(N_{T},N_{G}) \circ -} \\
\AAut(T_2,G_2,\partial_2) \ar[r]_{-\circ(N_{T},N_{G})}& **[r]\Hom((T_1,G_1,\partial_1),(T_2,G_2,\partial_2))}
\end{tabular}
\end{equation}
commutes. The trace of the object $p \in P_1$ along the upper-left and the lower-right sides of the diagram (\ref{1-morphism_of_actions_2}) provides respectively two 1-morphisms of crossed modules $(N_T \circ \theta^1_p, N_G \circ \sigma^1_p)$ and $(\theta^2_{\alpha_P(p)} \circ N_T,\sigma^2_{\alpha_P(p)} \circ N_G)$ where $X_{1\circ}(p)=(\theta^1_p,\sigma^1_p)$ and $X_{2\circ}(\alpha_P(p))=(\theta^2_{\alpha_P(p)}, \sigma^2_{\alpha_P(p)})$. They are equal due to the commutativity of (\ref{1-morphism_of_actions_2}). Hence, the relations:
\begin{align}
\label{1-morphism_of_actions_4} \theta^2_{\alpha_P(p)} \circ N_T&=N_T \circ \theta^1_p,\\
\label{1-morphism_of_actions_4bis} \sigma^2_{\alpha_P(p)} \circ N_G &=N_G \circ \sigma^1_p.
\end{align}
Similarly, by tracing the morphism $(m,p) \in M_1 \rtimes P_1$ along the upper-left and the lower-right sides of the diagram (\ref{1-morphism_of_actions_2}), we obtain the relation
\begin{equation}\label{1-morphism_of_actions_5}
N_T \circ X_{1\circ}(m,1) \circ \sigma^1_p=X_{2\circ}(\alpha_M(m),1) \circ \sigma^2_{\alpha_P(p)} \circ N_G.
\end{equation}

The actions of crossed modules on crossed modules and their morphisms form a category denoted by $\AXMod$.

\begin{remark}\label{rem:semi_direct_cross_product}
Given an action of $(M,P,\mu)$ on $(T,G,\partial)$ by a  crossed module morphism of the form (\ref{action_of_crossed_modules}), we can form the semidirect products $T \rtimes M$ and $G \rtimes P$. If $\rho=(\theta, \sigma)$, then $P$ acts on $G$ via $\sigma$ whereas $M$ acts on $T$ via $\theta \circ \mu$. The map $(\partial,\mu):T \rtimes M \ra G \rtimes P$ is a group homomorphism, as well. Moreover, there exists an action of $G \rtimes P$ over $T \rtimes M$ defined by  
\begin{equation}\label{action_of_cross_product}
{}^{(x,p)}(t,m)=\left({}^{x}({}^pt)(\varepsilon({}^pm)x^{-1}),{}^pm\right),
\end{equation}
where ${}^pt=\theta_p \circ \mu(t)$ and under this action the group homomorphism $(\partial,\mu)$ is a crossed module. We denote this crossed module by $(T \rtimes M, G \rtimes P, \partial \times \mu)$ and call it \emph{semidirect cross product}.
\end{remark}

\section{Whitehead Sequences}
In this section, we first remind from \cite{MR3579003} the notion of Whitehead sequences in general and we recall the motivating example behind it. Later, we define Whitehead sequences in the category of actions of crossed modules $\AXMod$.

\subsection{Recall on Whitehead Sequences}
Let $\oneC$ be a category and $\oneD$ be a pointed category.  Let $(I,R,J)$ be an ordered triple of functors 
\begin{equation}\label{triple_functors}
\xymatrix@1{\oneC \ar@<1.5ex>[r]^{I} \ar@<-1.5ex>[r]_{J}&\oneD \ar[l]|{R},}
\end{equation}
satisfying $I \circ R =\id_{\oneD}=J \circ R$. 

A \emph{Whitehead sequence} in $\oneC$ with respect to $(I,R,J)$ is a triple $(A,u,v)$ where $A$ is an object of $\oneC$ and $u:A \ra R \circ I(A)$ and $v: R \circ J(A) \ra A$ are morphisms in $\oneC$ satisfying the conditions
\begin{equation}\label{whitehead_sequence}
I(u)=\id_{I(A)}, ~ J(v)=\id_{J(A)}, ~ I(v)=J(u).
\end{equation}

Let us have a look at the case of groups. Let $\oneD$ be the category of groups $\grp$ and let $\oneC$ be the category of group actions on groups $\gact$ whose 
\begin{itemize}
\item objects are pairs $(X,B)$ where $B$ is a group considered as a category with one object and $X$ is a functor $X: B \ra \grp$ which sends the unique object of $G$ to $H$;
\item morphisms $(X_1,B_1) \ra (X_2,B_2)$ are pairs $(f,g)$ where $g:B_1 \ra B_2$ is a group homomorphism and $f:X_1 \ra X_2 \circ g$ is a natural transformation. 
\end{itemize}
Next, we define the functors 
\begin{align*}
I: \xymatrix@1{\gact \ar[r] & \grp},& ~ (X,B) \mapsto B;\\
J: \xymatrix@1{\gact \ar[r]& \grp}, &~ (X,B) \mapsto X_{\circ}:=H;\\
R: \xymatrix@1{\grp \ar[r] & \gact}, &~ B \mapsto (\overline{B},B);
\end{align*}
where the functor $\overline{G}:G \ra \grp$ corresponds to the conjugation action of $G$. It is clear that the triple $(I,R,J)$ satisfies the relations $I \circ R = \id_{\grp} =J \circ R$. In \cite{MR3579003}, it is shown that the ordered triple $(I,R,J)$ is an action system of $\gact$ over $\grp$ and a Whitehead sequence given by a collection $((X,B),u,v)$ with $u:(X,B) \ra (\overline{B},B)$ and $v:(\overline{H},H) \ra (X,B)$ morphisms of actions corresponds to a crossed module $H \ra B$ as dictated by the conditions (\ref{whitehead_sequence}).

\subsection{Whitehead Sequences in $\AXMod$}
We consider the ordered triple of functors $(I,R,J)$ defined as follows:
\begin{align}
\label{asf1}I: \xymatrix@1{\AXMod \ar[r] & \xmod},& ~ (X,(M,P,\mu)) \mapsto (M,P,\mu);\\
\label{asf2}J: \xymatrix@1{\AXMod \ar[r]& \xmod}, &~ (X,(M,P,\mu)) \mapsto X(\circ):=(T,G,\partial);\\
\label{asf3}R: \xymatrix@1{\xmod \ar[r] & \AXMod}, &~ (M,P,\mu) \mapsto (\Con,(M,P,\mu)).
\end{align}
These functors clearly satisfy the relations $I \circ R = \id_{\xmod} =J \circ R$. By definition, a Whitehead sequence in $\AXMod$ with respect to $(I,R,J)$ is a collection $((X,(M,P,\mu)),(U,(h,k)),(V,(f,g)))$ where $(X,(M,P,\mu))$ is an action of $(M,P,\mu)$ on $X(\circ):=(T,G,\partial)$, and $(U,(h,k))$ and $(V,(f,g))$ are the morphisms of actions given by
\begin{align}\label{Whitehead_seq_1}
(V,(f,g)):\xymatrix@1{(\Con,(T,G,\partial)) \ar[r] & (X,(M,P,\mu))}\\
(U,(h,k)):\xymatrix@1{(X,(M,P,\mu)) \ar[r] & (\Con,(M,P,\mu))}
\end{align}
so that $I(U,(h,k))=(h,k)=\id_{(M,P,\mu)}$, $J(V,(f,g))=V_{\circ}=\id_{(T,G,\partial)}$, and $I(V,(f,g))=(f,g)=U_{\circ}=J(U,(h,k))$. To be more precise, a Whitehead sequence in $\AXMod$ is a morphism of crossed modules $(f,g)$
\begin{equation}\label{crossed_square}
\begin{tabular}{c}
\xymatrix@!=1cm{T \ar[r]^{f} \ar[d]_{\partial} & M \ar[d]^{\mu}\\G \ar[r]_{g}& P}
\end{tabular}
\end{equation}
satisfying the relations
\begin{align}
\label{whitehead_1} \sigma_{g(x)}(x')&=xx'x^{-1};\\
\label{whitehead_2} \theta_{g(x)}(t)&={}^xt;\\
\label{whitehead_3} X_{\circ}(f(t),1)(x)&=t\!\ {}^x(t^{-1});\\
\label{whitehead_4} g(\sigma_p(x))&=pg(x)p^{-1};\\
\label{whitehead_5} f(\theta_p(t))&={}^{p}[f(t)];\\
\label{whitehead_6} f(X_{\circ}(m,1)(x))&=m \!\ {}^{g(x)}(m^{-1})
\end{align}
where $(\psi_p,\varphi_p):=\Con_{\circ}(p)$ is same as in Example \ref{ex:conjugation} and $(\theta_p,\sigma_p):=X_{\circ}(p)$. The relations (\ref{whitehead_1}) -- (\ref{whitehead_6}) are translations of the relations  (\ref{1-morphism_of_actions_4}) -- (\ref{1-morphism_of_actions_5}) for $(V,(f,g))$ and $(U,(h,k))$. We claim that the morphism of crossed modules (\ref{crossed_square}) with the relations (\ref{whitehead_1}) -- (\ref{whitehead_6}) is a crossed square. We define a map $\phi: M \times G \ra T$ by $\phi(m,x):=X_{\circ}(m,1)(x)$ and an action of $P$ on $G$ by ${}^px:=\sigma_p(x)$, and on $T$ by ${}^pt:=\theta_p(t)$. The relations (\ref{whitehead_1}) -- (\ref{whitehead_6})  can be rewritten as follows:
\begin{align}
\label{whitehead_1b} {}^{g(x)}x'&=xx'x^{-1};\\
\label{whitehead_2b} {}^{g(x)}t&={}^xt;\\
\label{whitehead_3b} \phi(f(t),x)&=t\!\ {}^x(t^{-1});\\
\label{whitehead_4b} g({}^{p}x)&=pg(x)p^{-1};\\
\label{whitehead_5b} f({}^pt)&={}^{p}[f(t)];\\
\label{whitehead_6b} f(\phi(m,x))&=m \!\ {}^{g(x)}(m^{-1})
\end{align}
We verify the crossed square axioms (CS1) -- (CS9). (\ref{whitehead_5b}) is the $P$-equivariance of $f$. As $(\theta_p,\sigma_p)$ is an automorphism of the crossed module $(T,G,\partial)$, $\partial({}^pt):=\partial(\theta_p(t))=\sigma_p(\partial(t)):={}^p\partial(t)$, therefore $\partial$ is $P$-equivariant, as well (CS1). The relations (\ref{whitehead_1b}) and (\ref{whitehead_4b}) imply that $g$ is a crossed module. The relation $g(\partial({}^pt))=g({}^p(\partial(t)))=pg(\partial(t))p^{-1}$ that follows from (\ref{whitehead_4b})  and from the $P$-equivariance of $\partial$ together with the relation ${}^{g(\partial(t))}t'={}^{\partial(t)}t'={}^tt'$ that follows from (\ref{whitehead_2b}) and the fact that $\partial$ is a crossed module imply that $g \circ \partial$ is a crossed module (CS2). (CS3) is exactly (\ref{whitehead_6b}). (CS5) follows straight from (\ref{whitehead_2b}) and (\ref{whitehead_3b}). As $X_{\circ}(m,1)$ is the morphism of automorphisms $(\id_T,\id_G) \Ra (\theta_{\mu(m)},\sigma_{\mu(m)})$ of the crossed module $(T,G,\partial)$, it satisfies the relations $\partial(X_{\circ}(m,1)(x))x=\sigma_{\mu(m)}(x)$ and $X_{\circ}(m,1)(\partial(t))t=\theta_{\mu(m)}(t)$ which imply (CS4) and (CS6), respectively. (CS7) follows from the definition of the horizontal composition of 2-morphisms of crossed modules $X_{\circ}(m_0,1)*X_{\circ}(m_1,1)$ as given in Lemma \ref{2nd_def_of_horizontal_composition} and (CS8) from (\ref{2_morphism_of_crossed_modules_3}) and (\ref{whitehead_2b}). Since by the proof of Proposition \ref{prop:action_as_functor} $(X_{\circ}(-,1),X_{\circ}(-)):(M,P,\mu) \ra \AAut(T,G,\partial)$ is a morphism of crossed modules, $X_{\circ}({}^pm,1)={}^{X_{\circ}(p)}[X_{\circ}(m,1)]:=\theta_p\circ X_{\circ}(m,1)\circ \sigma_p^{-1}$. Then $\phi({}^pm,{}^px)=X_{\circ}({}^pm,1)({}^px)=\theta_p\circ X_{\circ}(m,1)\circ \sigma_p^{-1} \circ \sigma_p(x)={}^p[X_{\circ}(m,1)(x)]={}^p\phi(m,x)$. Hence, (CS9) and our claim:
\begin{proposition}\label{prop:whitehead_seq}
A Whitehead sequence in $\AXMod$ with respect to $(I,R,J)$ defined by (\ref{asf1}), (\ref{asf2}), and (\ref{asf3}) is a crossed square. 
\end{proposition}

In Theorem \ref{thm:whitehead_crossed_square}, we will see that the notions of Whitehead sequences in $\AXMod$ and crossed squares are in fact equivalent. Let us define first the category of Whitehead sequences. A \emph{morphism of Whitehead sequences} in $\AXMod$ 
\[((X_1,(M_1,P_1,\mu_1)),(U_1,(h_1,k_1)),(V_1,(f_1,g_1))) \ra ((X_2,(M_2,P_2,\mu_2)),(U_2,(h_2,k_2)),(V_2,(f_2,g_2))),\]
is given by a morphism of actions $(N,(\alpha_M,\alpha_G)):(X_1,(M_1,P_1,\mu_1)) \ra (X_2,(M_2,P_2,\mu_2))$ so that the diagram commutes:
\begin{equation}\label{morph_Whitehead}
\begin{tabular}{c}
\xymatrix@C=1.75cm{(\Con,(T_1,G_1,\partial_1)) \ar[r]^{(V_1,(f_1,g_1))} \ar[d]_{(N',(N_T,N_G))} & (X_1,(M_1,P_1,\mu_1)) \ar[r]^{(U_1,(\id_{M_1},\id_{P_1}))} \ar[d]|{(N,(\alpha_M,\alpha_P))} & (\Con,(M_1,P_1,\mu_1)) \ar[d]^{(N'',(\alpha_M,\alpha_P))} \\
(\Con,(T_2,G_2,\partial_2)) \ar[r]_{(V_2,(f_2,g_2))} & (X_2,(M_2,P_2,\mu_2)) \ar[r]_{(U_2,(\id_{M_2},\id_{P_2}))} & (\Con,(M_2,P_2,\mu_2))}
\end{tabular}
\end{equation}
We observe that a given morphism of Whitehead sequences leads to a commutative diagram of morphisms of crossed modules
\begin{equation}\label{morph_Whitehead_1}
\begin{tabular}{c}
\xymatrix{(T_1,G_1,\partial_1) \ar[r]^{(f_1,g_1)} \ar[d]_{(N_T,N_G)} & (M_1,P_1,\mu_1)\ar[d]^{(\alpha_M,\alpha_P)}\\
(T_2,G_2,\partial_2) \ar[r]_{(f_2,g_2)} & (M_2,P_2,\mu_2)}
\end{tabular}
\end{equation}
so that the morphisms $(N_T,N_G)$ and $(\alpha_M,\alpha_G)$ are compatible with the actions.  Whitehead sequences and their morphisms form a category denoted by $\Wseq$.

\begin{theorem}\label{thm:whitehead_crossed_square}
The category of Whitehead sequences $\Wseq$ in $\AXMod$ and the category of crossed squares $\Xsq$ are isomorphic.
\end{theorem}
\begin{proof}
The discussion before the Theorem explains the construction of a functor from $\Wseq$ to $\Xsq$. To complete the proof, we shall construct a functor from $\Xsq$ to $\Wseq$. That is,  assume given a crossed square by the diagram (\ref{crossed_square_prep}) with map $\phi: M \times G \ra T$, we shall construct $(X,(M,P,\mu))$ an action of $(M,P,\mu)$ on $X(\circ)=(T,G,\partial)$ and we shall show that
\begin{align}
\label{from_crossed_sq_to_Whitehead_seq_1}(V,(f,g)):\xymatrix@1{(\Con,(T,G,\partial)) \ar[r] & (X,(M,P,\mu))}\\
\label{from_crossed_sq_to_Whitehead_seq_2}(U,(\id_M,\id_P)):\xymatrix@1{(X,(M,P,\mu)) \ar[r] & (\Con,(M,P,\mu))}
\end{align}
with $V_{\circ}=(\id_T,\id_G,\partial)$ and $U_{\circ}=(f,g)$ are morphisms of actions.

From \cite{MR1087375}[Theorem 3], we know that $(M,P,\mu)$ acts on $(T,G,\partial)$. This action is given by the crossed square morphism $(\varepsilon,\rho):(M,P,\mu) \ra (D(G,T),\Aut(T,G,\partial),\Delta)$. For any $p \in P$,  $\rho(p)=(\theta_p,\sigma_p) \in \Aut(T,G,\partial)$ is defined by $\theta_p(t)={}^pt$ and $\sigma_p(x)={}^px$. For any $m \in M$, $\varepsilon(m): G \ra T$ is defined by $\varepsilon(m)(x)=\phi(m,x)$. We can describe this action using a strict 2-functor $X: \underline{\underline{(M,P,\mu)}} \ra \Xmod$ whose construction is depicted in the proof of Proposition \ref{prop:action_as_functor}: $X$ sends the unique object $\circ$ of $\underline{\underline{(M,P,\mu)}}$ to $(T,G,\partial)$, a morphism $p$ to $(\theta_p,\sigma_p)$ an automorphism of $(T,G,\partial)$, and a 2-morphism $(m,p)$ to $\varepsilon(m) \circ \sigma_p$ a 2-morphism of automorphisms $(\theta_p,\sigma_p) \Ra (\theta_{\mu(m)p},\sigma_{\mu(m)p})$.

Showing that (\ref{from_crossed_sq_to_Whitehead_seq_1}) and (\ref{from_crossed_sq_to_Whitehead_seq_2}) are morphisms of actions comes down to showing that the collection of relations (\ref{whitehead_1b}) -- (\ref{whitehead_6b}) hold: (\ref{whitehead_1b}) and  (\ref{whitehead_2b}) are the consequences of (CS-1) and (CS-2). Similarly so are (\ref{whitehead_4b}) and (\ref{whitehead_5b}). (\ref{whitehead_3b}) follows from (CS-5), where (\ref{whitehead_6b}) from (CS-3).
\end{proof}

\section{Action Systems}
In this section, we remind from \cite{MR3579003} the definition of action systems. We show that the ordered triple $(I,R,J)$ defined in Proposition \ref{prop:whitehead_seq} is an action system of $\AXMod$ over $\xmod$ which satisfies the conditions of Theorem \ref{thm_intro}. 
\subsection{Recall on Action Systems}
Let $(I,R,J)$ be the ordered triple of functors (\ref{triple_functors}). A morphism $\alpha: A \ra B$ in $\oneC$ is called $I$-\emph{cartesian} if for any morphism $\alpha':A'\ra B$ in $\oneC$ and any morphism $f: I(A') \ra I(A)$ in $\oneD$ with $I(\alpha) \circ f=I(\alpha')$, there exists a unique morphism $\beta:A' \ra A$ in $\oneD$ so that $\alpha \circ \beta=\alpha'$. 

A \emph{patch} in $\oneD$ is a cospan $X \overset{k}{\ra} Z \overset{s}{\la} Y$ so that $k$ and $s$ are jointly epimorphic, and there exists a morphism $p: Z \ra Y$ satisfying $p \circ s=\id_Y$ and $p \circ k = 1$. We represent the above patch by the collection $(X,Y,Z,k,s,p)$. A patch $(X,Y,Z,k,s,p)$ is called \emph{exact} if the morphism $k:X \ra Z$ is the kernel of the morphism $p: Z \ra Y$. An $(I,J)-$\emph{organic morphism} is a morphism $\alpha: A\ra B$ in $\oneC$ so that $I(B) \cong J(B)$ and the cospan 
\begin{equation*}
J(\alpha):J(A) \ra J(B) \cong I(B) \la I(A):I(\alpha),
\end{equation*}
is an exact patch in $\oneD$.

The triple $(I,R,J)$ is an \emph{action system} of $\oneC$ over $\oneD$ if the following conditions are satisfied:
\begin{enumerate}[(\textrm{AS}-1)]
\item \label{AS_1} the functor $I$ is a fibration, that is for any object $B$ in $\oneC$ and any morphism $f:X \ra I(B)$ in $\oneD$ there exists a cartesian morphism  $\alpha:A \ra B$ in $\oneC$ so that $I(\alpha)=f$, and if $\alpha$ is a $I$-cartesian, then $J(\alpha)$ is an isomorphism in $\oneD$;
\item \label{AS_2} for every object $A$ in $\oneC$, there exists an object $Y$ in $\oneD$ and an $(I,J)$-organic morphism $f:A \ra R(Y)$ in $\oneC$ which is also universal;
\item \label{AS_3}the \emph{L-condition} holds, that is for every diagram of solid arrows
\begin{equation}\label{L_condition}
\begin{tabular}{c}
\xymatrix{RJ(E) \ar@{-->}[r]^{g'} \ar[dr]_{g} & E \ar@<-0.75ex>[d]_{\alpha}\ar@<0.75ex>@{<-}[d]^{\beta} \ar@{-->}[r]^{f'} &RI(E) \\&A \ar[ur]_{f}&}
\end{tabular}
\end{equation}
satisfying the conditions $I(\beta)=I(f)$, $J(\alpha)=J(g)$, $I(\alpha) \circ J(f)=I(g) \circ J(\beta)$ and $\alpha \circ \beta = \id_A$, if $\alpha$ is cartesian and $f$ is organic, then there exists a unique Whitehead sequence $(f',g')$ such that $\alpha \circ g'=g$ and $f'\circ \beta=f$.
\end{enumerate} 
\subsection{Action System of $\AXMod$ over $\xmod$}
We verify that the ordered triple $(I,R,J)$ defined by (\ref{asf1}), (\ref{asf2}), and (\ref{asf3}) satisfy the conditions (AS-\ref{AS_1}), (AS-\ref{AS_2}), and (AS-\ref{AS_3}) of an action system. Let us first define the $I$-cartesian morphisms in $\AXMod$ and the $(I,J)$-organic morphisms in $\xmod$. We say that the cospan of morphisms of crossed modules
\begin{equation}\label{diag:cospan}
\begin{tabular}{c}
\xymatrix{(X_1,X_2,\delta_X) \ar[r]^{(k_1,k_2)} & (Y_1,Y_2,\delta_Y) & (Z_1,Z_2,\delta_Z) \ar[l]_{(s_1,s_2)}}
\end{tabular}
\end{equation}
is 
\begin{itemize}
\item \emph{jointly epimorphic} if the cospans $(k_1,s_1)$ and $(k_2,s_2)$ are jointly epimorphic,
\item a \emph{patch} if it is jointly epimorphic, and there exists a morphism of crossed modules \\ $(p_1,p_2):(Y_1,Y_2,\delta_Y) \ra (Z_1,Z_2,\delta_Z)$ so that $(p_1\circ k_1, p_2\circ k_2)=(1,1)$ and $(p_1 \circ s_1, p_2 \circ s_2)=\id_{(Z_1,Z_2,\delta_Z)}$,
\item an \emph{exact patch} if it is a patch and $(k_1,k_2)$ is the kernel of $(p_1,p_2)$.
\end{itemize}

Let $(X_i,(M_i,P_i,\mu_i))$ be an action of $(M_i,P_i,\mu_i)$ on $(T_i,G_i,\partial_i)$ for $i=1,2$. A morphism of actions $(N,(\alpha_M,\alpha_P)):(X_1,(M_1,P_1,\mu_1)) \ra (X_2,(M_2,P_2,\mu_2))$ is $(I,J)$-organic if 
\begin{itemize}
\item there exists $(r,q): (T_2,G_2,\partial_2)\ra (M_2,P_2,\mu_2)$ an isomorphism of crossed modules;
\item the cospan 
\begin{equation}\label{diag:cospan_action}
\begin{tabular}{c}
\xymatrix@1{(T_1,G_1,\partial_1) \ar[r]^{N_{\circ}} & (T_2,G_2,\partial_2) \ar[r]^{(r,q)}_{\simeq} & (M_2,P_2,\mu_2) &\ar[l]_{(\alpha_M,\alpha_P)} (M_1,P_1,\mu_1)}
\end{tabular}
\end{equation}
is an exact patch.
\end{itemize}

\textit{Verification of (AS-\ref{AS_1})}: Let $(X_2,(M_2,P_2,\mu_2))$ be an action and $(\alpha_M,\alpha_P):(M_1,P_1,\mu_1) \ra (M_2,P_2,\mu_2)$ be a morphism of crossed modules. Then $(\alpha_M,\alpha_P)$ can be lifted to $(1,(\alpha_M,\alpha_P))$ a cartesian morphism of actions $(X_2 \circ A,(M_2,P_2,\mu_2)) \ra (X_2,(M_2,P_2,\mu_2))$ where $A$ is the functor associated to $(\alpha_M,\alpha_P)$. In fact, given a morphism of actions $(N,(\gamma_M,\gamma_P)):(X_3,(M_3,P_3,\mu_3))\ra (X_2,(M_2,P_2,\mu_2)$ and a morphism of crossed modules $(\beta_M,\beta_P):(M_3,P_3,\mu_3) \ra (M_1,P_1,\mu_1)$ so that $(\alpha_M,\alpha_P) \circ (\beta_M,\beta_P)=(\gamma_M,\gamma_G)$, there exists a morphism of actions $(N,(\beta_M,\beta_P)):(X_3,(M_3,P_3,\mu_3)) \ra (X_2 \circ A,(M_1,P_1,\mu_1))$ so that $(1,(\alpha_M,\alpha_P)) \circ (N,(\beta_M,\beta_P))=(N,(\gamma_M,\gamma_G))$. 

Moreover, let $(N,(\alpha_M,\alpha_P)):(X_1,(M_1,P_1,\mu_1)) \ra (X_2,(M_2,P_2,\mu_2))$ be a cartesian morphism. Then there exists a unique action morphism $(H,(\beta_M,\beta_P)):(X_2,(M_2,P_2,\mu_2)) \ra (X_1,(M_1,P_1,\mu_1))$ so that $(H,(\beta_M,\beta_P)) \circ (N,(\alpha_M,\alpha_P))=(1,\id_{(M_2,P_2,\mu_2)})$. Therefore, $J((H,(\beta_M,\beta_P)) \circ (N,(\alpha_M,\alpha_P)))=J(1,\id_{(M_2,P_2,\mu_2)})$, that is $J((H,(\beta_M,\beta_P))) \circ J((N,(\alpha_M,\alpha_P)))=\id$. Hence, $J((N,(\alpha_M,\alpha_P)))$ is an isomorphism.

\textit{Verification of (AS-\ref{AS_2})}: Let $(X,(M,P,\mu))$ define an action of $(M,P,\mu)$ on $X_{\circ}=(T,G,\partial)$ and $(T \rtimes M, G \rtimes P, \partial \times \mu)$ be the crossed module defined in Remark \ref{rem:semi_direct_cross_product}. We claim that the morphism of actions
\begin{equation}\label{canonical_injections}
(N,(\lambda_M,\lambda_P)):\xymatrix@1{(X,(M,P,\mu)) \ar[r] &  (\Con,(T \rtimes M, G \rtimes P, \partial \times \mu)),}
\end{equation}
where $(\lambda_M,\lambda_P)$ is the canonical injection $(M,P,\mu) \ra (T \rtimes M, G \rtimes P, \partial \times \mu)$, $\Lambda$ is the functor associated to $(\lambda_M,\lambda_P)$, and $N: X \Ra \Con \circ \Lambda$ is the transformation that sends the only object $\circ$ of $\underline{\underline{(M,P,\mu)}}$ to $(\gamma_T,\gamma_G)$ the canonical injection $(T,G,\partial) \ra (T \rtimes M, G \rtimes P, \partial \times \mu)$ is organic. First, we shall show that $N$ is strict. To that end, it is enough to prove that the diagram 
\begin{equation}\label{diag:semi_strict_product}
\begin{tabular}{c}
\xymatrix{\underline{(M,P,\mu)} \ar[r]^{X_{\circ}} \ar[d]_{\Con_{\circ} \circ \Lambda} & \AAut(T,G,\partial) \ar[d]^{(\gamma_T,\gamma_G) \circ -} \\
\AAut(T \rtimes M,G \rtimes P,\partial\times \mu) \ar[r]_{-\circ(\gamma_T,\gamma_G)}& **[r]\Hom((T,G,\partial),(T \rtimes M,G \rtimes P,\partial\times \mu))}
\end{tabular}
\end{equation}
commutes. Let  $p$ be an object in $\underline{(M,P,\mu)}$. $\Con_{\circ} \circ \Lambda$ sends $p$ to the automorphisms $\psi_{(1,p)}:T \rtimes M \ra T \rtimes M$ and $\varphi_{(1,p)}: G \rtimes P \ra G \rtimes P$ which are defined at $(t,1)$ for any $t \in T$ by: 
\begin{align*}
\psi_{(1,p)}(t,1)&={}^{(1,p)}(t,1)=({}^pt,1),\\
\varphi_{(1,p)}(x,1)&=(1,p)(x,1)(1,p^{-1})=({}^px,1).
\end{align*}
These definitions follow from Example \ref{ex:conjugation} and from the action (\ref{action_of_cross_product}). We note that the notations ${}^pt$ and ${}^px$ mean $\theta_p(t)$ and $\sigma_p(x)$, respectively, where $X_{\circ}(p)=(\theta_p,\sigma_p)$. Hence, the commutativity of (\ref{diag:semi_strict_product}). Showing that (\ref{canonical_injections}) is an organic morphism comes down to showing that the cospan diagram
\begin{equation}
\begin{tabular}{c}
\xymatrix@!=1cm{T \ar[r]^(0.4){\gamma_T} \ar[d]_{\partial} & T \rtimes M \ar[d]|{\partial\times \mu}&M\ar[d]^{\mu} \ar[l]_(0.4){\lambda_M}\\
G \ar[r]_(0.4){\gamma_G}&G\rtimes P&P \ar[l]^(0.4){\lambda_P}}
\end{tabular}
\end{equation}
in $\xmod$ is an exact patch. We only detail the fact that $(\gamma_T,\gamma_G)$ and $(\lambda_M,\lambda_P)$ are jointly epimorphic leaving the rest to the reader. Let $(E,F,\delta)$ be a crossed module and  
\begin{equation}\label{jointly_epimorphic}
\begin{tabular}{c}
\xymatrix@!=1cm{T \rtimes M \ar[d]_{\partial_\times \mu}\ar@<0.5ex>[r]^{\kappa_1} \ar@<-0.5ex>[r]_{\iota_1} & E \ar[d]^{\delta}\\
G \rtimes P\ar@<0.5ex>[r]^{\kappa_2} \ar@<-0.5ex>[r]_{\iota_2} & F}
\end{tabular}
\end{equation}
be crossed module morphisms so that $(\kappa_1,\kappa_2) \circ (\gamma_T,\gamma_G)=(\iota_1,\iota_2) \circ (\gamma_T,\gamma_G)$ and $(\kappa_1,\kappa_2) \circ (\lambda_M,\lambda_P)=(\iota_1,\iota_2) \circ (\lambda_M,\lambda_P)$. Then for any $(t,m) \in T \rtimes M$
\begin{align*}
\kappa_1(t,m)&=\kappa_1((t,1)(1,m))=\kappa_1(\gamma_T(t)\lambda_M(m))=\kappa_1(\gamma_T(t))\kappa_1(\lambda_M(m))=\iota_1(\gamma_T(t))\iota_1(\lambda_M(m))=\iota_1(t,m).
\end{align*}
Similarly, we show that $\kappa_2=\iota_2$.
 
\textit{Verification of (AS-\ref{AS_3})}: In \cite{MR3579003}, it is noted that in semi-abelian categories the L-condition is equivalent to the ``Smith is Huq" condition defined in \cite{MR2899723}. As the category of crossed modules is semi-abelian, proving the condition (AS-3) is equivalent to proving that the category of crossed modules has (SH) property. In fact, as the category of crossed modules is algebraically coherent by \cite[Proposition 4.18]{zbMATH06520725}  and all algebraically coherent categories have (SH) property (see \cite[Theorem 6.18] {zbMATH06520725}), the category of crossed modules has (SH), as well. In this paper, we verify this condition by proving in detail the existence of the non solid arrows in the diagram below:

\begin{equation}\label{L_condition_2}
\begin{tabular}{c}
\xymatrix@R=1.5cm@C=2cm{(\Con,(T,G,\partial)) \ar@{-->}[r]^(0.4){(N_3',(f',g'))} \ar[dr]_{(N_3,(f,g))}& (X',(M',P',\mu')) \ar@{-->}[r]^(0.55){(N_4',(q_M',q_P'))} \ar@<-0.75ex>[d]_(0.45){(N_1,(r_M,r_P))}\ar@<0.75ex>@{<-}[d]^(0.45){(N_2,(s_M,s_P))} & (\Con,(M',P',\mu'))\\
&(X,(M,P,\mu)) \ar[ur]_{(N_4,(q_M,q_P))}&}
\end{tabular}
\end{equation}
with conditions
\begin{enumerate}[(L-1)]
\item\label{L1}$I(N_2,(s_M,s_P))=I(N_4,(q_M,q_P))$;
\item \label{L2}$J(N_1,(r_M,r_P))=J(N_3,(f,g))$;
\item \label{L3}$I(N_1,(r_M,r_P)) \circ J(N_4,(q_M,q_P))=I(N_3,(f,g))\circ J(N_2,(s_M,s_P))$;
\item \label{L4}$(N_1,(r_M,r_P)) \circ (N_2,(s_M,s_P))=\id_{(X,(M,P,\mu))}$;
\item \label{L5}$(N_1,(r_M,r_P))$ is cartesian;
\item \label{L6}$(N_4,(q_M,q_P))$ is organic.
\end{enumerate}
The diagram (\ref{L_condition_2}) is the diagram (\ref{L_condition}) with $A=(X,(M,P,\mu))$ the action of $(M,P,\mu)$ on $(T,G,\partial)$ and $E=(X',(M',P',\mu'))$ the action of $(M',P',\mu')$ on $(T',G',\partial')$. As $(N_1,(r_M,r_P))$ is cartesian, it can be lifted to the morphism $(\id,(r_M,r_P)):(X \circ R,(M',P',\mu')) \ra (X,(M,P,\mu))$ where $R$ is the strict 2-functor associated to $(r_M,r_P)$. From (L-\ref{L2}), $(N_3)_{\circ}=\id_{(T,G,\partial)}$. From (L-\ref{L4}), $N_2=\id$ and $(r_M,r_P) \circ (s_M,s_P)=(\id_M,\id_P)$. From (L-\ref{L1}), $(s_M,s_P)=(q_M,q_P)$. From (L-\ref{L3}), $(r_M,r_P) \circ (\kappa,\iota) =(f,g)$ where $(N_4)_{\circ}=(\kappa,\iota):(T,G,\partial) \ra (M',P',\mu')$. In summary, under the conditions (L-1) -- (L-6), the diagram (\ref{L_condition_2}) becomes
\begin{equation*}\label{L_condition_3}
\begin{tabular}{c}
\xymatrix@R=1.5cm@C=2cm{(\Con,(T,G,\partial)) \ar@{-->}[r]^(0.4){(N_3,(\kappa,\iota))} \ar[dr]_{(N_3,(f,g))}& (X \circ R,(M',P',\mu')) \ar@{-->}[r]^(0.55){(N_4,(\id_{M'},\id_{P'})} \ar@<-0.75ex>[d]_{(\id,(r_M,r_P))}\ar@<0.75ex>@{<-}[d]^{(\id,(s_M,s_P))} & (\Con,(M',P',\mu'))\\
&(X,(M,P,\mu)) \ar[ur]_{(N_4,(s_M,s_P))}&}
\end{tabular}
\end{equation*}
where $(r_M,r_P) \circ (\kappa,\iota)=(f,g)$, $(r_M,r_P) \circ (s_M,s_P)=(\id_M,\id_P)$, and $(N_4)_{\circ}=(\kappa,\iota)$.

We discuss in details why the morphism $(N_4,(\id_{M'},\id_{P'}))$ is in $\AXMod$. In order $(N_4,(\id_{M'},\id_{P'}))$ to be a morphism in $\AXMod$,  the diagram
\begin{equation}\label{1-morphism_of_actions_2bis}
\begin{tabular}{c}
\xymatrix{\underline{(M',P',\mu')} \ar[d]_{\Con_{\circ}} \ar[r]^{X_{\circ} \circ R}& \underline{\AAut(T,G,\partial)} \ar[d]^{(\kappa,\iota) \circ -} \\
\underline{\AAut(M',P',\mu')} \ar[r]_{-\circ (\kappa,\iota)}& **[r]\Hom((T,G,\partial),(M',P',\mu'))}
\end{tabular}
\end{equation}
which is the diagram (\ref{1-morphism_of_actions_2}) adapted to our case, should commute.  That is, the relations
\begin{align}
\label{as1} \kappa \circ \theta_{r_P(p')}(t)&=\psi_{p'}\circ \kappa(t)\mathrm{~for~ all~} p' \in P', t \in T;\\
\label{as2} \iota \circ \sigma_{r_P(p')}(x)&=\varphi_{p'}\circ \iota(x)\mathrm{~for~ all~} p' \in P', x \in G;\\
\label{as3} \kappa \circ X_{\circ}(r_M(m')) \circ \sigma_{r_P(p')}(x)&=\Con_{\circ}(m') \circ \sigma_{p'} \circ \iota(x)\mathrm{~for~ all~} p' \in P', m'\in M', x \in G;
\end{align}
where $\Con_{\circ}(p')=(\psi_{p'},\varphi_{p'})$ as in Example (\ref{ex:conjugation}) and $X_{\circ}(p')=(\theta_{p'},\sigma_{p'})$ which are obtained as a result of tracing objects and morphisms in the diagram (\ref{1-morphism_of_actions_2bis}) should hold. For greater eligibility, in the rest of the calculations, we write for any $x \in G$, $t \in T$, and $p \in P$, $^{p}x=\sigma_p(x)$ and $^{p}t=\theta_p(t)$.  

Before we verify the relations (\ref{as1}), (\ref{as2}), and (\ref{as3}), we note the following technical points. From the fact that the morphism $(N_4,(s_M,s_P))$ is an organic morphism, we can identify the crossed module $(M',P',\mu')$ by the direct product $(T \rtimes M, G \rtimes P, \partial \times \mu)$ and the morphism $(N_4,(s_M,s_P))$ by the canonical injections. Therefore, any $p' \in P'$ can be written as $(y,q) \in G \rtimes P$ and any $m \in M'$ as $(u,n) \in T \rtimes M$. We define $r_P(y,q):=g(y)q$ for any $(y,q) \in G \rtimes P$ and $r_M(u,n):=f(u)n$ for any $(u,n) \in T \rtimes M$. The other technical point is that for any $m \in M$, the inverse of the regular derivation $X_{\circ}(m,1)$ is the regular derivation $^{m}X_{\circ}(m^{-1},1)$ and therefore for any $x \in G$, $[X_{\circ}(m,1)(x)]^{-1}={}^{m}[X_{\circ}(m^{-1},1)(x)]=X_{\circ}(m^{-1},1)({}^mx)$.
 
Now, we are ready to show that the relations (\ref{as1}), (\ref{as2}), and (\ref{as3}) are satisfied. Verifications of (\ref{as1}) and (\ref{as2}) are similar. So we only provide details for (\ref{as1}): 
\begin{equation*}
\kappa \circ \theta_{r_P((y,q))}(t)=({}^{g(y)q}t,1)=({}^{yq}t,1)={}^{(y,q)}(t,1)=\psi_{(y,q)}\circ \kappa(t).
\end{equation*}
To verify (\ref{as3}), we observe that the left hand side of (\ref{as3}) \begin{align*}
 \kappa \circ X_{\circ}(r_M((u,n)),1) \circ \sigma_{r_P((y,q))}(x)&=\left(X_{\circ}(f(u)n,1)({}^{g(y)q}x),1\right)\\
&=\left(X_{\circ}(f(u),1)*X_{\circ}(n,1)({}^{g(y)q}x),1\right)\\
&=\left(X_{\circ}(f(u),1)({}^{\mu(n)g(y)q}x)X_{\circ}(n,1)({}^{g(y)q}x),1\right) (\mathrm{from ~ (\ref{horizontal_composition})}) \\
&=\left(\Con_{\circ}(u,1)({}^{\mu(n)g(y)q}x)X_{\circ}(n,1)({}^{g(y)q}x),1\right) ((N_3,(f,g))~\mathrm{is~a~morphism~of~action})\\ 
&=\left(u{}^{({}^{\mu(n)g(y)q}x)}u^{-1}X_{\circ}(n,1)({}^{g(y)q}x),1\right) ~(\mathrm{definition~of~}\Con_{\circ})
 \end{align*}
 and the right hand side of (\ref{as3})
 \begin{align*}
 \Con_{\circ}((u,n)) \circ \sigma_{(y,q)} \circ \iota(x)&=\left(u,n\right)\!\ ^{({}^{g(y)q}x,1)}\left({}^{n^{-1}}u^{-1},n^{-1}\right) (\mathrm{definition~of~}\Con_{\circ})\\
&=\left(u,n\right)\left({}^{{}^{g(y)q}x}({}^{n^{-1}}u^{-1})\left[X_{\circ}(n^{-1},1)({}^{g(y)q}x)\right]^{-1},n^{-1}\right) (\mathrm{from~(\ref{action_of_cross_product})})\\
&=\left(u{}^{({}^{\mu(n)g(y)q}x)}u^{-1}\left[X_{\circ}(n^{-1},1)({}^{\mu(n)g(y)q}x)\right]^{-1},1\right)~(\mathrm{operation~in~} T\rtimes M)\\
&=\left(u{}^{({}^{\mu(n)g(y)q}x)}u^{-1}X_{\circ}(n,1)({}^{g(y)q}x),1\right)~(\mathrm{inverse~of~a~derivation})
\end{align*}
 are equal. This proves that:

\begin{proposition}\label{prop:action_system}
The ordered triple $(I,R,J)$ defined by (\ref{asf1}), (\ref{asf2}), and (\ref{asf3}) is an action system in $\AXMod$ over $\xmod$.
\end{proposition}

Let $(f,g): (T,G,\partial) \ra (M,P,\mu)$ be a split epimorphism in $\xmod$ with the section $(q,r):(M,P,\mu) \ra (T,G,\partial)$. We claim that $(M,P,\mu)$ acts on $(\ker f,\ker g, \overline{\partial})$ where $\overline{\partial}$ is the restriction of $\partial$ over $\ker f$. This action is given by the crossed module morphism
\begin{equation}\label{from_point_to_action}
\begin{tabular}{c}
\xymatrix@!=1cm{M \ar[d]_{\mu} \ar[r]^{\varepsilon} & **[r]D(\ker f,\ker g)\ar[d]^{\Delta}\\P \ar[r]_{\rho} & **[r]\Aut(\ker f,\ker g,\overline{\partial})}
\end{tabular}
\end{equation}
where for any $m \in M$, $\varepsilon(m):\ker g \ra \ker f, x \mapsto q(m) \!\ {}^{x}q(m)^{-1}$ and for any $p \in P$, $\rho(p)=(\theta_p,\sigma_p)$ with $\theta_p:T \ra T, t \mapsto ^{r(p)}t$ and $\sigma_p: G \ra G, x \mapsto ^{r(p)}x$.

Reciprocally, given an action of $(M,P,\mu)$ on $(T,G,\partial)$, we can form the semidirect product crossed module $(T \rtimes M,G \rtimes P, \partial \times \mu)$ as described in Remark \ref{rem:semi_direct_cross_product}. Then the  projection of the semidirect product on to $(M,P,\mu)$
\begin{equation}\label{from_action_to_point}
\begin{tabular}{c}
\xymatrix@!=1cm{T \rtimes M \ar[d]_{\partial_\times \mu}\ar[r] & M \ar[d]^{\mu}\\
G \rtimes P\ar[r] & P}
\end{tabular}
\end{equation}
is a split epimorphism with the section being canonical injection. This gives an isomorphism between the category of points in $\xmod$ and $\AXMod$. Moreover, the functor $R: \xmod \ra \AXMod$ in (\ref{asf3}) has left adjoint $L: \AXMod \ra \xmod$ defined by $(X,(M,P,\mu)) \mapsto (T \rtimes M,G\rtimes P,\partial \times \mu)$ which is compatible with the action system. Hence, we deduce Theorem \ref{thm:whitehead_crossed_square} as a consequence of Theorem \ref{thm_intro}.

\bibliographystyle{plain}

\end{document}